\documentclass{article}
\usepackage{amsmath}
\usepackage{amssymb}
\usepackage{amsthm}
\usepackage[colorlinks=true, bookmarks=true, pdfstartview=FitH, linktocpage=true, linkcolor = magenta, citecolor = blue]{hyperref}

\hypersetup{
	colorlinks = true,
	linkcolor = magenta,
	citecolor = blue,
}
\newtheorem*{prbs}{Problem $\mathcal{P}_{shallow,\boldsymbol{\theta}}$}
\newtheorem*{prbp}{Problem $\mathcal{P}_{plate}$}
\newtheorem{thm}{Theorem}
\newtheorem{lem}{Lemma}

\newtheorem{rem}{Remark\/}
\newtheorem{rems}[rem]{Remarks\/}

\begin{document}
	
	\title{Asymptotic behavior of a nonlinear shallow shell model when the shell becomes a plate}
	\author{Trung Hieu Giang
		\footnote{corresponding author}
		\footnote{Addresses of Trung Hieu Giang: Department of Mathematical Analysis, Faculty of Mathematics and Physics, Charles University, Praha, Czech Republic.  Institute of Mathematics, Vietnam Academy of Science and Technology, 18 Hoang Quoc Viet, Hanoi, Vietnam. \textit{Email address: trung-hieu.giang@matfyz.cuni.cz}
		} 
		\  and
		Ngoc Quynh Nguyen
		\footnote{Addresses of Ngoc Quynh Nguyen: Department of Mathematics, Hanoi National University of Education, Hanoi, Vietnam. \textit{Email address: ngocquynh.ng25@gmail.com}
	}}
	\date{\today}
	\maketitle
	
	\begin{abstract}
		This paper studies a nonlinear shallow shell model proposed by Donnell, Vlasov, Mushtari, Galimov, and Koiter. More specifically, we address the question concerning the asymptotic behavior of minimizing solutions. Our result can be applied to general applied forces. Thus, it substantially extends the one given in \cite{oana2} whereby the tangential components of the applied forces are assumed to vanish.
	\end{abstract}
	
	\
	
	
	\
	
	
	
	\section{Notation}
	\label{s1}
	
	Throughout this paper, Greek indices and exponents range in the set $\{1,2\}$, while Latin indices and exponents range in the set $\{1,2,3\}$, unless they are used for indexing sequences or otherwise specified in the text. The summation convention with respect to repeated indices and exponents is also used. 
	
	Strong and weak convergences in any normed vector space are respectively denoted 
	$\to\text{ and }\rightharpoonup.$
	
	The notation $y=(y_\alpha)$ denotes a generic point in $\mathbb{R}^2$ and partial derivatives, in the classical or distributional sense, are denoted $\partial_\alpha:= \partial/\partial y_\alpha$ and $\partial_{\alpha\beta}:= \partial^2/\partial y_\alpha \partial y_\beta$.
	
	Vector and matrix fields are denoted by boldface letters. The Euclidean norm, the inner product, and the vector product of two vectors $\boldsymbol{u}$ and $\boldsymbol{v}$ in $ \mathbb{R}^3$ are respectively denoted $|\boldsymbol{u}|$, $\boldsymbol{u} \cdot \boldsymbol{v}$, and $\boldsymbol{u} \wedge \boldsymbol{v}$. The notation $(c_{\alpha\beta})$ designates the $2\times 2$ matrix whose component at its $\alpha$-row and $\beta$-column is $c_{\alpha\beta}$.
	
	The usual norm of the Lebesgue space $L^p (\omega)$, $1 \leq p \leq \infty$, is denoted by $\| \cdot \|_{L^p(\omega)}$. The notation $L^p (\omega; \mathbb{R}^{m\times n})$ denotes, for every positive integer $m$ and $n$, the space of matrix fields $\boldsymbol{A}=(A_{ij}): \omega \to \mathbb{R}^{m \times n}$ with components $A_{ij}$ in the Lebesgue space $L^p(\omega)$. This space of matrix-valued functions is equipped with the norm
	\begin{equation*}
		\|\boldsymbol{A}\|_{L^p(\omega)} := \Big( \sum\limits_{i=1}^m\sum\limits_{j=1}^n \|A_{ij}\|^p_{L^p(\omega)}\Big)^{1/p} \textrm{ for all } \boldsymbol{A} \in L^p (\omega; \mathbb{R}^{m\times n}).
	\end{equation*}
	Note that we use the same notation for the norm of scalar, vector, and matrix-valued functions.
	
	The usual norm of the Sobolev space $W^{m,p}(\omega)$, $m \in \mathbb{N}^*$, $1 \leq p \leq \infty$, is denoted by $\| \cdot \|_{W^{m,p}(\omega)}$. The notations $H^1(\omega)$ and $H^2(\omega)$ respectively denote the spaces $W^{1,2}(\omega)$ and $W^{2,2}(\omega)$. We also denote by $H^1_0 (\omega)$ and $H^2_0 (\omega)$ respectively the closures of $\mathcal{D}(\omega)$ in $H^1(\omega)$ and $H^2(\omega)$.
	
	In this paper, by shell we mean a three-dimensional body whose natural state (that is, a stress-free configuration of the body) is a ``surface with positive thickness'' in the three-dimensional Euclidean space. This means that a shell is a subset of ${\mathbb{R}}^3$ of the form
	$$
	\overline{\boldsymbol{\Theta}(\Omega)},
	$$
	where 
	$$
	\Omega:=\omega\times (-\varepsilon, \varepsilon)
	$$
	and 
	$$
	\boldsymbol{\Theta} (y, x_3) := \boldsymbol{\theta}(y) + x_3 \boldsymbol{a}_{3,\boldsymbol{\theta}}(y), \ (y,x_3)\in\Omega,
	$$
	for some given bounded and connected open set $\omega\subset{\mathbb{R}}^2$ whose boundary is Lipschitz-continuous in the sense of Adams \& Fournier \cite{ada}, real number $0<\varepsilon\leq \varepsilon_0$, and immersion $\boldsymbol{\theta}: \overline{\omega} \to \mathbb{R}^3$ of class ${\mathcal{C}}^2$ with a unit vector field $\boldsymbol{a}_{3,\boldsymbol{\theta}}:\overline{\omega} \to \mathbb{R}^3$ defined by
	\begin{equation}
		\label{s1e1}
		\boldsymbol{a}_{3,\boldsymbol{\theta}}(y) := \dfrac{{\partial}_1\boldsymbol{\theta}(y) \wedge {\partial}_2\boldsymbol{\theta}(y)}{\left|{\partial}_1\boldsymbol{\theta}(y) \wedge {\partial}_2\boldsymbol{\theta}(y) \right|} \textrm{ for all } y\in\overline{\omega}.
	\end{equation}
	
	The assumption that $\boldsymbol{\theta}$ is an immersion of class ${\mathcal{C}}^2$ means that $\boldsymbol{\theta}\in{\mathcal{C}}^2(\omega; {\mathbb{R}}^3)$ and the two vector fields $\boldsymbol{a}_{\alpha,\boldsymbol{\theta}}:={\partial}_\alpha\boldsymbol{\theta}$ are linearly independent at every $y\in\omega$, so that the unit normal vector field $\boldsymbol{a}_{3,\boldsymbol{\theta}}$ in \eqref{s1e1} is well defined and of class $\mathcal{C}^1$.
	
	The real number $\varepsilon>0$ represents the half-thickness of the shell, and $S:=\boldsymbol{\theta}(\overline{\omega})$ represents the middle surface of the shell. The upper bound $\varepsilon_0$ of $\varepsilon$ is chosen sufficiently small compared with the diameter of $\omega$, so that $\boldsymbol{\Theta}$ is itself an immersion of class ${\mathcal{C}}^1$ from $\overline{\Omega}$ into $\mathbb{R}^3$ (the existence of such a small $\varepsilon_0$ is proved in \cite{ciarlet-G}).
	
	Note that $\boldsymbol{a}_{3,\boldsymbol{\theta}}(y)$ is the positively oriented unit vector normal to the surface $S$ at the point $\boldsymbol{\theta}(y)$, while the two vectors $\boldsymbol{a}_{\alpha,\boldsymbol{\theta}} (y)$ form a basis in the tangent plane to $S$ at the same point. The three vectors $\boldsymbol{a}_{i,\boldsymbol{\theta}} (y)$ constitute the \textit{covariant basis} in $\mathbb{R}^3$ at the point $\boldsymbol{\theta}(y)$, while the three vectors $\boldsymbol{a}^i_{\boldsymbol{\theta}} (y)$ uniquely defined by the relations
	$$
	\boldsymbol{a}^i_{\boldsymbol{\theta}} (y) \cdot \boldsymbol{a}_{j,\boldsymbol{\theta}} (y) = \delta^i_j,
	$$
	where $\delta^i_j$ is the Kronecker symbol, constitute the \textit{contravariant basis} at the point $\boldsymbol{\theta}(y) \in S$. As a consequence, any vector field $\boldsymbol{\zeta} : \omega \to \mathbb{R}^3$ can be decomposed over either of these bases as
	$$
	\boldsymbol{\zeta} = \zeta_i \boldsymbol{a}^i_{\boldsymbol{\theta}} = \zeta^i \boldsymbol{a}_{i,\boldsymbol{\theta}},
	$$
	for some functions $\zeta_i: \omega \to \mathbb{R}$ and $\zeta^i : \omega \to \mathbb{R}$.
	
	The area element of the surface $S$ is denoted by $\sqrt{a_{\boldsymbol{\theta}}(y)}dy$, where
	\begin{equation}\label{s1e2}
		a_{\boldsymbol{\theta}}:=|\partial_1 \boldsymbol{\theta} \wedge \partial_2 \boldsymbol{\theta}|^2 \text{ in } \omega.  
	\end{equation}
	
	The covariant components $a_{\alpha\beta,\boldsymbol{\theta}} \in \mathcal{C}^1 (\overline{\omega})$ and $b_{\alpha\beta,\boldsymbol{\theta}} \in \mathcal{C}^0 (\overline{\omega})$ of respectively the first and second fundamental forms of $S = \boldsymbol{\theta} (\overline{\omega})$ are defined by 
	$$a_{\alpha\beta,\boldsymbol{\theta}} := \boldsymbol{a}_{\alpha,\boldsymbol{\theta}} \cdot \boldsymbol{a}_{\beta,\boldsymbol{\theta}} \textrm{ and }  b_{\alpha\beta,\boldsymbol{\theta}} :=\boldsymbol{a}_{3,\boldsymbol{\theta}} \cdot \partial_{\alpha\beta}\boldsymbol{\theta} \textrm{ in } \overline{\omega}.$$
	
	The contravariant components of the first fundamental form are the components $a^{\alpha\beta}_{\boldsymbol{\theta}} \in \mathcal{C}^1 (\overline{\omega})$ of the inverse matrix 
	$$
	(a^{\alpha\beta}_{\boldsymbol{\theta}}(y)):=(a_{\alpha\beta,\boldsymbol{\theta}}(y))^{-1}, \ y\in\overline{\omega}. 
	$$
	
	The Christoffel symbols $\Gamma^\sigma_{\alpha\beta,\boldsymbol{\theta}} \in \mathcal{C}^0 (\overline{\omega})$ of $S$ are defined by
	$$
	\Gamma^\sigma_{\alpha\beta,\boldsymbol{\theta}} := \boldsymbol{a}^\sigma_{\boldsymbol{\theta}} \cdot \partial_\beta \boldsymbol{a}_{\alpha,\boldsymbol{\theta}} \textrm{ in } \overline{\omega}.
	$$
	
	The Gaussian curvature of the surface $S=\boldsymbol{\theta}(\overline{\omega})$ is defined by 
	$$
	K_{\boldsymbol{\theta}}:=\det \left(a^{\alpha\sigma}_{\boldsymbol{\theta}}b_{\sigma\beta, \boldsymbol{\theta}}\right) \in {\mathcal{C}}^0(\overline{\omega}).
	$$
	
	
	\section{The nonlinear shallow shell model}
	\label{s2}
	
	The lower-dimensional elasticity theories play an important role in mathematical analysis and numerical computations. On the one hand, they allow simpler computational approaches; on the other hand, they often possess a richer variety of mathematical results. In recent decades, formal asymptotic analysis, based on various additional mathematical and mechanical assumptions, has led to a significant number of different two-dimensional shell theories, such as the nonlinear membrane shell, flexural shell, Koiter shell, etc. (for a more complete list of shell theories, we refer to \cite{cvol3} and the references therein). However, to date, many theoretical aspects of these models are still not fully understood.
	
	This paper is concerned with a nonlinear shallow shell model, which is a two-dimensional theory proposed by  L.H. Donnell, V.Z. Vlasov, K.M. Mushtari \& K.Z. Galimov, and W.T. Koiter \cite{don, koi, mush, vla}. In this context, a shell is called ``shallow" when the Gaussian curvature $K_{\boldsymbol{\theta}}$ of its middle surface is sufficiently small. The nonlinear shallow shell model is applied to the shells made of an elastic material whose behavior is modeled by two elastic coefficients, denoted $\lambda\geq0$ and $\mu>0$, via the two-dimensional elasticity tensor $(a_{\boldsymbol{\theta}}^{\alpha\beta\sigma\tau})$, where
	\begin{equation}
		\label{s2e1}
		a_{\boldsymbol{\theta}}^{\alpha\beta\sigma\tau} := \dfrac{4\lambda\mu}{\lambda+2\mu} a_{\boldsymbol{\theta}}^{\alpha\beta}a_{\boldsymbol{\theta}}^{\sigma\tau} + 2\mu (a_{\boldsymbol{\theta}}^{\alpha\sigma}a_{\boldsymbol{\theta}}^{\beta\tau}+a_{\boldsymbol{\theta}}^{\alpha\tau}a_{\boldsymbol{\theta}}^{\beta\sigma}).
	\end{equation}
	Notice that this tensor is positive-definite in the sense that there exists a constant $c_e (\omega, \boldsymbol{\theta}, \lambda, \mu) > 0$ such that
	\begin{equation}\label{s2e2}
		c_e(\omega, \boldsymbol{\theta}, \lambda, \mu)\sum\limits_{\alpha,\beta} |t_{\alpha\beta}|^2 \leq  a_{\boldsymbol{\theta}}^{\alpha\beta\sigma\tau}(y) t_{\sigma\tau}t_{\alpha\beta},
	\end{equation}
	for all $y \in \overline{\omega}$ and all symmetric matrices $(t_{\alpha\beta})$ (see, e.g., \cite[Theorem 3.3-2]{cvol3}).
	
	For simplicity, our analysis in this paper is restricted to the case of totally clamped shells. Then, the nonlinear shallow shell model is often formulated as a minimization problem as follows: 
	
	\begin{prbs}\label{shallow}
		Find $\boldsymbol{u} = (u_i) \in \boldsymbol{V}(\omega)$ that minimizes the energy
		\begin{align}\label{s2e3}
			\nonumber
			J_{\boldsymbol{\theta}}(\boldsymbol{\eta}) := &\frac{1}{2} \int\limits_{\omega} \Big\{  \frac{\varepsilon^3}{3} a_{\boldsymbol{\theta}}^{\alpha \beta \sigma \tau} F_{\sigma\tau}^{\boldsymbol{\theta}}(\boldsymbol{\eta})F_{\alpha\beta}^{\boldsymbol{\theta}}(\boldsymbol{\eta}) + \varepsilon a_{\boldsymbol{\theta}}^{\alpha \beta \sigma \tau} E^{\boldsymbol{\theta}}_{\sigma\tau}(\boldsymbol{\eta})E^{\boldsymbol{\theta}}_{\alpha\beta}(\boldsymbol{\eta}) \Big\} \sqrt{a_{\boldsymbol{\theta}}} \, dy \\
			& - \int\limits_{\omega} p^i \eta_i \sqrt{a_{\boldsymbol{\theta}}} \, dy, \textrm{ for every } \boldsymbol{\eta} \in \boldsymbol{V}(\omega),
		\end{align}
		where 
		$$
		\boldsymbol{V}(\omega) := H^1_0(\omega) \times H^1_0(\omega) \times H^2_0(\omega),
		$$
		$$
		E^{\boldsymbol{\theta}}_{\alpha\beta}(\boldsymbol{\eta}) := \dfrac{1}{2} \big(\partial_\beta\eta_{\alpha} + \partial_\alpha\eta_{\beta} \big) - \Gamma^\sigma_{\alpha\beta,\boldsymbol{\theta}}\eta_\sigma - b_{\alpha\beta,\boldsymbol{\theta}}\eta_3 + \dfrac{1}{2} \partial_\alpha\eta_{3}\partial_\beta\eta_{3},
		$$
		$$
		F^{\boldsymbol{\theta}}_{\alpha\beta}(\boldsymbol{\eta}) := \partial_{\alpha\beta} \eta_3 - \Gamma^\sigma_{\alpha\beta,\boldsymbol{\theta}} \partial_\sigma \eta_3,
		$$
		and the functions $p_i \in L^2(\omega)$ are the components of the density $\boldsymbol{p} = (p_i)$ of the applied force. Here, each admissible field $\boldsymbol{\eta} = (\eta_i) \in \boldsymbol{V}(\omega)$ corresponds to an admissible displacement $\eta_i \boldsymbol{a}^i_{\boldsymbol{\theta}}$ of the middle surface $\boldsymbol{\theta}(\overline{\omega})$.  
	\end{prbs} 
	
	Over the last several decades, this nonlinear shallow shell model has been studied by several authors, with a primary focus on the existence, uniqueness, and asymptotic behavior of solutions. Regarding the Euler-Lagrange equation associated with the functional $J_{\boldsymbol{\theta}}$, by using the theory of pseudo-monotone operators, Bernadou \& Oden \cite{bern} have established the existence of a weak solution if the tangential components of the applied forces and the curvature of the middle surface are sufficiently small. They also proved that the weak solution found is unique if the applied forces are sufficiently small. Figueiredo \cite{figueiredo1} has obtained a local existence theorem for the boundary value problem associated with problem \hyperref[shallow]{$\mathcal{P}_{shallow, \boldsymbol{\theta}}$} using the Inverse Function Theorem, provided that the applied forces are sufficiently small. Iosifescu \cite{oana1} has shown that this model has at least one solution when the applied tangential components of the applied forces vanish by introducing an Airy stress function in curvilinear coordinates and using the same method as Ciarlet \& Rabier \cite{ciarlet2} for establishing the existence of a solution to von K\'arm\'an equations. Additionally, under the same assumptions, in \cite{oana2}, she also studied the behavior of the solutions to this model when the shell becomes a plate. Regarding the minimization problem \hyperref[shallow]{$\mathcal{P}_{shallow, \boldsymbol{\theta}}$}, Ciarlet \& Iosifescu \cite{ciarlet2} via an intrinsic approach have established the existence of a minimizer under the additional assumptions that the Gaussian curvature and the applied forces are small enough. However, we would like to emphasize that the above results rely on various additional assumptions on the applied forces, and thus, the case of general forces remains an open question. 
	
	By contrast, an existence theorem has been established by Rabier \cite{rab} for general applied forces in the special case where 
	$$
	\boldsymbol{\theta} (y) = (y,0) \textrm{ in } \overline{\omega},
	$$
	i.e., when the shell is a plate. In particular, it can be shown that the following minimization problem has at least one minimizer:
	
	\begin{prbp}\label{plate}
		Find $\boldsymbol{u} = (u_i) \in \boldsymbol{V}(\omega)$ that minimizes the energy
		\begin{align}\label{s2e4}
			\nonumber
			J(\boldsymbol{\eta}) := &\frac{1}{2} \int\limits_{\omega} \Big\{  \frac{\varepsilon^3}{3} a_0^{\alpha \beta \sigma \tau} \partial_{\sigma\tau}\eta_3 \partial_{\alpha\beta}\eta_3 + \varepsilon a_0^{\alpha \beta \sigma \tau} E^0_{\sigma\tau}(\boldsymbol{\eta})E^0_{\alpha\beta}(\boldsymbol{\eta}) \Big\}  \, dy \\
			& - \int\limits_{\omega} p^i \eta_i  \, dy, \textrm{ for every } \boldsymbol{\eta} \in \boldsymbol{V}(\omega),
		\end{align}
		where 
		$$
		a_0^{\alpha\beta\sigma\tau} := \dfrac{4\lambda\mu}{\lambda+2\mu} \delta^{\alpha\beta}\delta^{\sigma\tau} + 2\mu (\delta^{\alpha\sigma}a^{\beta\tau}+\delta^{\alpha\tau}\delta^{\beta\sigma}),
		$$
		where $\delta^{\alpha\beta}$ is the Kronecker symbol, and
		$$
		E^0_{\alpha\beta}(\boldsymbol{\eta}) := \dfrac{1}{2}(\partial_\alpha \eta_\beta + \partial_\beta \eta_\alpha) + \dfrac{1}{2}\partial_\alpha \eta_3 \partial_\beta \eta_3.
		$$
	\end{prbp} 
	It can be easily seen that this is exactly the formulation for the nonlinearly elastic plate model (cf. Ciarlet \cite{cvol2}). Thus, the existence of a minimizer to the problem \hyperref[plate]{$\mathcal{P}_{plate}$} can be achieved via a crucial ingredient, namely, the rigidity of the plates (see \cite{rab} for the case of the clamped plates, and also see \cite{giamar} for more general boundary conditions). This ingredient will also play an important role in our analysis, and it will be fully presented in the next section.
	
	In this paper, we aim to answer the question of the asymptotic behavior of the problem \hyperref[shallow]{$\mathcal{P}_{shallow, \boldsymbol{\theta}}$} without any restriction on the applied forces. In particular, we will show that, under the impact of any arbitrary applied force, when the shallow shell is close enough to the plate, the problem \hyperref[shallow]{$\mathcal{P}_{shallow, \boldsymbol{\theta}}$} converges in a suitable sense to problem \hyperref[plate]{$\mathcal{P}_{plate}$}. Therefore, our result considerably extends the one given by Iosifescu \cite{oana2}.
	
	\section{Main result}
	\label{s3}
	
	In this section, we will study the asymptotic behavior of the problem \hyperref[shallow]{$\mathcal{P}_{shallow, \boldsymbol{\theta}}$} when the shell becomes a plate. To begin with, we need the following, which is the rigidity property of plates:
	
	\begin{lem}[The rigidity property of plates]
		\label{s3l1}
		Let $\omega\subset {\mathbb{R}}^2$ be a bounded and connected open set with Lipschitz-continuous boundary. Then the set $\boldsymbol{V}(\omega)$ satisfies the following rigidity property:
		\begin{equation}\label{s3e1}\aligned{} 
			{\boldsymbol{u}} \in \boldsymbol{V}(\omega) \text{ and } (E^0_{\alpha\beta}({\boldsymbol{u}}))={\boldsymbol{0}} \text{ in } \omega \text{ imply } {\boldsymbol{u}}={\boldsymbol{0}} \text{ in } \omega.
			\endaligned\end{equation}
	\end{lem}
	
	\begin{proof}
		The proof of this property is first given by Rabier \cite{rab}, and then by Giang \& Mardare \cite{giamar} with a simpler approach. For readers' convenience, we will present the details here. We will follow the proof given in \cite{giamar}.
		
		Let ${\boldsymbol{u}}\in \boldsymbol{V}(\omega)$ satisfy $(E^0_{\alpha\beta}({\boldsymbol{u}}))={\boldsymbol{0}}$ in $\omega$. Using the boundary conditions satisfied by $u_{\alpha}$ and integrating by parts, we deduce from $(E^0_{\alpha\beta}({\boldsymbol{u}}))={\boldsymbol{0}}$ that 
		\begin{equation*}
			\int_\omega \partial_\alpha u_3\partial_\beta u_3 dy=0.
		\end{equation*}
		In particular,  
		\begin{equation*}
			\int_\omega (\partial_\alpha u_3)^2dy=0,
		\end{equation*}
		which, combined with the boundary conditions satisfied by $u_3$, implies that
		\begin{equation*}
			u_3 = 0 \textrm{ in } \omega.
		\end{equation*}
		
		Consequently, the relation $(E^0_{\alpha\beta}({\boldsymbol{u}}))={\boldsymbol{0}}$ in $\omega$ implies that  
		\begin{equation*}
			\partial_\beta u_\alpha + \partial_\alpha u_\beta = 0 \textrm{ in } \omega,
		\end{equation*}
		which in turn implies that  
		\begin{equation*}
			\partial_{\alpha\beta}u_\sigma=0  \textrm{ in } \omega.  
		\end{equation*}
		The set $\omega$ being connected, the functions $u_\sigma$ must be affine. Since they vanish on the entire boundary of a bounded set, necessarily at three non-collinear points, they must vanish in $\omega$:  
		\begin{equation*}
			u_\alpha = 0 \textrm{ in } \omega.
		\end{equation*}
		Our proof is complete.		
	\end{proof}
	
	\begin{rem}
		\label{s3r1}
		The above rigidity property plays a crucial role in establishing an existence theorem for problem \hyperref[plate]{$\mathcal{P}_{plate}$} as it is a sufficient condition for the weak coercivity of the functional $J$ (see \cite{rab, giamar}). It is therefore natural to inquire whether an analogous property holds in the case of shallow shells. In particular, if one could prove that $\boldsymbol{u} \in \boldsymbol{V}(\omega)$ and
		\begin{equation}\label{s3e2}
			\dfrac{1}{2} (\partial_\beta u_{\alpha} + \partial_\alpha u_{\beta }) - \Gamma^\sigma_{\alpha\beta,\boldsymbol{\theta}} \eta_\sigma + \dfrac{1}{2}\partial_\alpha u_{3 }\partial_\beta u_{3} =0 \text{ in } \omega, \textrm{ for all } \alpha, \, \beta,
		\end{equation}
		imply $\boldsymbol{u}=\boldsymbol{0}$ in $\omega$, then an existence result for the nonlinear shallow shell model under general applied forces would follow in a manner similar to the case of plates. To date, however, this remains an open question.
	\end{rem}
	
	We also recall the following, which is the sequential weak lower semicontinuity of the functional $J_{\boldsymbol{\theta}}$:
	\begin{lem}
		\label{s3l2}
		Let $\omega\subset{\mathbb{R}}^2$ be a bounded and connected open set whose boundary is Lipschitz-continuous and let $\boldsymbol{\theta}$ be an immersion of class $\mathcal{C}^2$.  Then the functional $J_{\boldsymbol{\theta}}: \boldsymbol{V}(\omega) \to{\mathbb{R}}$ is sequentially weakly lower semicontinuous.
	\end{lem}
	
	\begin{proof}
		The proof is similar to the one for the nonlinearly elastic plate model (see, for example, \cite{cvol2}). We will present it here for readers' convenience. Since $\omega$ is a bounded Lipschitz domain in ${\mathbb{R}}^2$, the embeddings $H^2_0(\omega) \hookrightarrow W^{1,4}(\omega)$ and $H^1_0(\omega) \hookrightarrow L^2(\omega)$ are compact; cf. Adams \& Fournier \cite{ada}. Then, if a sequence ${\boldsymbol{u}}^n=(u^n_i)$, $n\in{\mathbb{N}}$, and an element ${\boldsymbol{u}}=(u_i)$ in the space $\boldsymbol{V}(\omega)$ satisfy  
		\begin{equation*}
			u_\alpha^n \rightharpoonup u_\alpha \text{ \ in } H^1_0(\omega) \text{ as } n\to \infty  
		\end{equation*}
		and  
		\begin{equation*}
			u_3^n \rightharpoonup u_3 \text{ \ in } H^2_0(\omega) \text{ as } n\to \infty,
		\end{equation*}
		then  
		\begin{equation*}
			u_\alpha^n \to u_\alpha \text{ \ in } L^2(\omega) \text{ when } n\to \infty.
		\end{equation*}
		and
		\begin{equation*}
			u_3^n \to u_3 \text{ \ in } W^{1,4}(\omega) \text{ when } n\to \infty.
		\end{equation*}
		
		Consequently,
		\begin{equation*}\aligned{}
			E^{\boldsymbol{\theta}}_{\alpha\beta}({\boldsymbol{u}}^n)& \rightharpoonup E^{\boldsymbol{\theta}}_{\alpha\beta}({\boldsymbol{u}})\text{ \ in } L^2(\omega) \text{ as } n\to \infty,\\
			F^{\boldsymbol{\theta}}_{\alpha\beta}({\boldsymbol{u}}^n)& \rightharpoonup F^{\boldsymbol{\theta}}_{\alpha\beta}({\boldsymbol{u}})\text{ \ in } L^2(\omega) \text{ as } n\to \infty.
			\endaligned\end{equation*}
		From this and the positive-definiteness of the elasticity tensor $(a_{\boldsymbol{\theta}}^{\alpha\beta\sigma\tau})$ given in \eqref{s2e2}, it is easy to see that
		\begin{equation*}
			J_{\boldsymbol{\theta}}({\boldsymbol{u}})\leq \liminf_{n\to\infty} J_{\boldsymbol{\theta}}({\boldsymbol{u}}^n).
		\end{equation*}
	\end{proof}
	
	The main result of this paper is as follows:
	
	\begin{thm}
		\label{s3t1}
		Let $\omega\subset{\mathbb{R}}^2$ be a bounded and connected open set whose boundary is Lipschitz-continuous, and let $\boldsymbol{\theta}_0 : \overline{\omega} \to \mathbb{R}^3$ be defined by
		$$
		\boldsymbol{\theta}_0 (y) := (y,0) \textrm{ for all } y \in \overline{\omega}.
		$$
		Let $\boldsymbol{p} = (p_i) \in (L^2(\omega))^3$ be the density of an arbitrary applied force. Then the following hold:
		\begin{itemize}
			\item[(i)] There exists a positive number $\delta_0$ such that, for every immersion $\boldsymbol{\theta}$ of class $\mathcal{C}^2$ satisfying 
			$$\| \boldsymbol{\theta} - \boldsymbol{\theta}_0\|_{\mathcal{C}^2(\overline{\omega})} < \delta_0,$$
			the problem \hyperref[shallow]{$\mathcal{P}_{shallow, \boldsymbol{\theta}}$} admits a solution $\boldsymbol{u}_{\boldsymbol{\theta}}$, and the family $(\boldsymbol{u}_{\boldsymbol{\theta}})_{\boldsymbol{\theta}}$ is bounded in $\boldsymbol{V}(\omega)$.
			\item[(ii)] For any sequence $(\boldsymbol{\theta}_n)_n$ converging to $\boldsymbol{\theta}_0$ in $\mathcal{C}^2(\overline{\omega})$, any weakly convergent subsequence of $(\boldsymbol{u}_{\boldsymbol{\theta}_n})_n$ converges strongly in $\boldsymbol{V}(\omega)$ to one of the solutions of problem \hyperref[plate]{$\mathcal{P}_{plate}$}. Consequently, if \hyperref[plate]{$\mathcal{P}_{plate}$} admits a unique solution $\boldsymbol{u}_0$, then the entire sequence $(\boldsymbol{u}_{\boldsymbol{\theta}_n})_n$ strongly converges to $\boldsymbol{u}_0$.
		\end{itemize}
	\end{thm}
	
	\begin{proof}
		\noindent(i) For readers' convenience, we will divide the proof into two steps.
		
		\noindent\textbf{Step 1:} For every immersion $\boldsymbol{\theta} \in \mathcal{C}^2(\overline{\omega})$, we define
		$$
		\mathcal{A}_{\boldsymbol{\theta}} := \big\{ \boldsymbol{\eta} = (\eta_i) \in \boldsymbol{V}(\omega); \, J_{\boldsymbol{\theta}}(\boldsymbol{\eta}) \leq 0\big\}.
		$$
		
		It is easy to see that $\mathcal{A}_{\boldsymbol{\theta}}$ is nonempty, since $\boldsymbol{0} \in \mathcal{A}_{\boldsymbol{\theta}}$. Next, we will prove that: There exist $\delta_0 > 0$ and $M_0 >0$ such that for every $\boldsymbol{\theta}$ satisfying
		$$
		\| \boldsymbol{\theta} - \boldsymbol{\theta}_0\|_{\mathcal{C}^2(\overline{\omega})} < \delta_0,
		$$
		the sets $\mathcal{A}_{\boldsymbol{\theta}}$ are bounded by $M_0$.
		
		We will prove by contradiction. Assume that for every $n \in \mathbb{N}$, there exists an immersion $\boldsymbol{\theta}_n \in \mathcal{C}^2 (\overline{\omega})$ satisfying
		$$
		\| \boldsymbol{\theta}_n - \boldsymbol{\theta}_0\|_{\mathcal{C}^2(\overline{\omega})} < \dfrac{1}{n},
		$$
		such that there exists $\boldsymbol{u}^n = (u_i^n) \in \mathcal{A}_{\boldsymbol{\theta}_n}$ satisfying
		$$
		\| \boldsymbol{u}^n\|_{\boldsymbol{V}(\omega)} = \lambda_n \geq n.
		$$
		Here
		$$
		\|\boldsymbol{u}\|_{\boldsymbol{V}(\omega)} := \|u_1\|_{H^1_0 (\omega)} + \|u_2\|_{H^1_0 (\omega)} + \|u_3\|_{H^2_0 (\omega)},
		$$
		for every $\boldsymbol{u} = (u_i) \in \boldsymbol{V}(\omega)$.
		
		It follows from the definition of $\mathcal{A}_{\boldsymbol{\theta}_n}$ that
		\begin{align}\label{s3e3}
			\nonumber
			&\frac{1}{2} \int\limits_{\omega} \Big\{\frac{\varepsilon^3}{3} a_{\boldsymbol{\theta}_n}^{\alpha \beta \sigma \tau} F^{\boldsymbol{\theta}_n}_{\sigma \tau} (\boldsymbol{u}^n) F^{\boldsymbol{\theta}_n}_{\alpha \beta}(\boldsymbol{u}^n) + \varepsilon a_{\boldsymbol{\theta}_n}^{\alpha \beta \sigma \tau} E^{\boldsymbol{\theta}_n}_{\sigma\tau}(\boldsymbol{u}^n)E^{\boldsymbol{\theta}_n}_{\alpha\beta}(\boldsymbol{u}^n) \Big\} \sqrt{a_{\boldsymbol{\theta}_n}} \, dy \\
			& \leq \int\limits_{\omega} p^i u^n_i \sqrt{a_{\boldsymbol{\theta}_n}} \, dy. 
		\end{align}
		
		Notice that $\boldsymbol{\theta}_n$ and $\boldsymbol{\theta}_0$ are immersions of class $\mathcal{C}^2$, hence it can be easily seen that there exists a number $N_1>0$ such that there exists a constant $C_e > 0$ so that, for every $n>N_1$, we have
		\begin{equation}\label{s3e4}
			C_e \sum\limits_{\alpha,\beta} |t_{\alpha\beta}|^2 \leq  a_{\boldsymbol{\theta}_n}^{\alpha\beta\sigma\tau}(y) t_{\sigma\tau}t_{\alpha\beta} \sqrt{a_{\boldsymbol{\theta}_n}(y)},
		\end{equation}
		for all $y \in \overline{\omega}$ and all symmetric matrices $(t_{\alpha\beta})$. Also, it can be proved that there exists a number $N_2 >0$ such that there exists a constant $C_1>0$ so that, for every $n>N_2$, we have
		\begin{equation}
			\label{s3e5}
			\sqrt{a_{\boldsymbol{\theta}_n}} \leq C_1 \textrm{ for all } y \in \overline{\omega}.
		\end{equation}
		Then, from \eqref{s3e3}-\eqref{s3e5} and together with the H\"older's inequality, we deduce that there exists a constant $C_2 >0$ such that
		\begin{align}
			\label{s3e6}
			\nonumber
			&\sum\limits_{\alpha,\beta} \big( \|F^{\boldsymbol{\theta}_n}_{\alpha \beta}(\boldsymbol{u}^n)\|^2_{L^2(\omega)} + \|E^{\boldsymbol{\theta}_n}_{\alpha\beta}(\boldsymbol{u}^n)\|^2_{L^2(\omega)}\big) \\
			& \leq C_2 \max_{i=1,2,3} \|p^i\|_{L^2(\omega)} \sum\limits_{i=1}^3 \|u_i^n\|_{L^2(\omega)},
		\end{align}
		for every $n > \max \{N_1, N_2\}$. 
		
		Next, let 
		$$
		\tilde{\boldsymbol{u}}^n = (\tilde{u}_i^n) := \dfrac{\boldsymbol{u}^n}{\lambda_n} \textrm{ for all } n,
		$$
		we have
		\begin{equation}
			\label{s3e7}
			\|\tilde{\boldsymbol{u}}^n\|_{\boldsymbol{V}(\omega)} = 1 \textrm{ for all } n. 
		\end{equation}
		Then, it follows from the continuity of the embeddings $H^1_0(\omega) \hookrightarrow L^2(\omega)$ and $H^2_0(\omega) \hookrightarrow H^1_0(\omega)$ that
		\begin{equation}
			\label{s3e8}
			\sum\limits_{\alpha=1}^2 \| \tilde{u}_\alpha^n\|_{L^2(\omega)} + \| \tilde{u}_3^n\|_{H^1_0(\omega)} \leq C_3 \textrm{ for all } n,
		\end{equation}
		where $C_3$ is some positive constant.
		
		By dividing both sides of \eqref{s3e6} by $\lambda_n^2$ and using \eqref{s3e8} together with the fact that $\lambda_n \to \infty$ as $n \to \infty$, we obtain
		\begin{equation}
			\label{s3e9}
			\lim\limits_{n\to\infty} \sum\limits_{\alpha,\beta} \|F^{\boldsymbol{\theta}_n}_{\alpha \beta}(\tilde{\boldsymbol{u}}^n)\|_{L^2(\omega)} = 0,
		\end{equation}
		and
		\begin{equation}
			\label{s3e10}
			\lim\limits_{n\to\infty} \sum\limits_{\alpha,\beta} \big\|\dfrac{1}{\lambda_n}E^{\boldsymbol{\theta}_n}_{\alpha\beta}(\boldsymbol{u}^n)\big\|_{L^2(\omega)} = 0.
		\end{equation}
		
		Notice that
		$$
		\lim\limits_{n\to\infty} \|\Gamma^\sigma_{\alpha\beta, \boldsymbol{\theta}_n} \|_{\mathcal{C}^0(\overline{\omega})} = 0,
		$$
		and 
		$$
		\lim\limits_{n\to\infty} \|b_{\alpha\beta, \boldsymbol{\theta}_n} \|_{\mathcal{C}^0(\overline{\omega})} = 0.
		$$
		
		Then, it follows from \eqref{s3e8}, \eqref{s3e9}, and \eqref{s3e10} that
		\begin{equation}
			\label{s3e11}
			\lim\limits_{n\to\infty} \sum\limits_{\alpha,\beta} \|\partial_{\alpha\beta} \tilde{u}^n_3\|_{L^2(\omega)} = 0,
		\end{equation}
		and
		\begin{equation}
			\label{s3e12}
			\lim\limits_{n\to\infty} \sum\limits_{\alpha,\beta} \|\partial_\alpha \tilde{u}^n_\beta + \partial_\beta \tilde{u}^n_\alpha + \lambda_n\partial_\alpha\tilde{u}^n_3\partial_\beta\tilde{u}^n_3\|_{L^2(\omega)} = 0.
		\end{equation}
		The Poincar\'e inequality and \eqref{s3e11} implies that
		\begin{equation}
			\label{s3e13}
			\lim\limits_{n\to\infty}\|\tilde{u}^n_3\|_{H^2_0(\omega)} = 0.
		\end{equation}
		
		Next, it follows from the fact that  $\|\Gamma^\sigma_{\alpha\beta, \boldsymbol{\theta}_n} \|_{\mathcal{C}^0(\overline{\omega})} \to 0$ as $n\to 0$ and the Poincar\'e inequality that there exists $N_3 >0$ such that 
		\begin{equation}
			\label{s3e14}
			\sum\limits_{\alpha,\beta} \|\partial_{\alpha\beta} v - \Gamma^\sigma_{\alpha\beta, \boldsymbol{\theta}_n} \partial_\sigma v\|_{L^2(\omega)} \geq \dfrac{1}{2} \sum\limits_{\alpha,\beta}\|\partial_{\alpha\beta} v\|_{L^2(\omega)},
		\end{equation}
		for all $n>N_3$ and for all $v \in H^2_0(\omega)$. By using this, \eqref{s3e8} and by dividing both sides of \eqref{s3e6} for $\lambda_n$, we obtain
		\begin{equation}
			\label{s3e15}
			\dfrac{1}{2} \sum\limits_{\alpha,\beta}\|\sqrt{\lambda_n}\partial_{\alpha\beta}  \tilde{u}^n_3\|_{L^2(\omega)} \leq C_2C_3 \max\limits_{i=1,2,3} \|p^i\|_{L^2(\omega)},
		\end{equation}
		for all $n> N_0 := \max \{N_1, N_2, N_3\}$. Again, as a consequence of the Poincar\'e inequality, we deduce that the sequence 
		$$
		(\sqrt{\lambda_n} \tilde{u}^n_3)_n 
		$$
		is bounded in $H^2_0(\omega)$. Then, there exists a subsequence (still denoted by) $(\sqrt{\lambda_n} \tilde{u}^n_3)_n$ and a function $v \in H^2_0(\omega)$, such that
		$$
		\sqrt{\lambda_n} \tilde{u}^n_3 \rightharpoonup v \textrm{ in } H^2_0(\omega),
		$$
		and
		\begin{equation}
			\label{s3e16}
			\sqrt{\lambda_n} \tilde{u}^n_3 \to v \textrm{ in } W^{1,4}(\omega),
		\end{equation}
		as $n \to \infty$.
		
		The equality \eqref{s3e7} implies that the sequences $(\tilde{u}_1^n)_n$ and $(\tilde{u}_2^n)_n$ are bounded in $H^1_0(\omega)$. Thus, there exist subsequences (still denoted by) $(\tilde{u}_1^n)_n$ and $(\tilde{u}_2^n)_n$ and functions $v_1, v_2 \in H^1_0(\omega)$ such that
		$$
		\tilde{u}_1^n \rightharpoonup v_1 \textrm{ and } \tilde{u}_2^n \rightharpoonup v_2 \textrm{ in } H^1_0(\omega)
		$$
		as $n \to \infty$. Hence, it follows from this, \eqref{s3e12}, and \eqref{s3e16} that
		$$
		\sum\limits_{\alpha,\beta} \| \partial_\alpha v_\beta + \partial_\beta v_\alpha + \partial_\alpha v \partial_\beta v\|_{L^2(\omega)} = 0.
		$$
		Then, by using Lemma \ref{s3l1}, we deduce that
		$$
		v_1 = v_2 = v =0.
		$$
		Thanks to this and \eqref{s3e16}, we can further deduce that
		$$
		\lim\limits_{n\to\infty} \sum\limits_{\alpha,\beta} \| \partial_\alpha \tilde{u}^n_\beta + \partial_\beta \tilde{u}^n_\alpha \|_{L^2(\omega)} = 0.
		$$
		As a consequence of Korn's inequality, we have
		\begin{equation}
			\label{s3e17}
			\tilde{u}_1^n \to 0 \textrm{ and } \tilde{u}_2^n \to 0 \textrm{ in } H^1_0(\omega)
		\end{equation}
		as $n \to \infty$.
		
		Now, by combining \eqref{s3e13} and \eqref{s3e17}, we obtain
		$$
		\lim\limits_{n\to \infty} \|\tilde{\boldsymbol{u}}^n \|_{\boldsymbol{V}(\omega)} = 0,
		$$
		which contradicts \eqref{s3e7}. Hence, our statement holds.
		
		\noindent\textbf{Step 2:} Let $\delta_0$ and $M_0$ be given as in the beginning of Step 1. This means that any minimizing sequence $(\boldsymbol{v}^n)_n$ of $J_{\boldsymbol{\theta}}$ is bounded, provided that 
		$$
		\|\boldsymbol{\theta} - \boldsymbol{\theta}_0\|_{\mathcal{C}^2(\overline{\omega}} < \delta_0.
		$$
		Thus, it follows from the sequential weak lower semicontinuity of the functional $J_{\boldsymbol{\theta}}$ (cf. Lemma \ref{s3l2}) that the problem \hyperref[shallow]{$\mathcal{P}_{shallow, \boldsymbol{\theta}}$} admits a solution $\boldsymbol{u}_{\boldsymbol{\theta}}$.
		
		Since
		$$
		0 = J_{\boldsymbol{\theta}}(\boldsymbol{0}) \geq J_{\boldsymbol{\theta}}(\boldsymbol{u}_{\boldsymbol{\theta}}),
		$$
		we have that
		$$
		\boldsymbol{u}_{\boldsymbol{\theta}} \in \mathcal{A}_{\boldsymbol{\theta}}.
		$$
		Now, notice that the sets $\mathcal{A}_{\boldsymbol{\theta}}$ are bounded by $M_0$, our desired result follows. Our proof is complete. 
		
		\,
		
		\noindent(ii) The proof will be divided into two steps for readers' convenience. 
		
		\noindent\textbf{Step 1:} Let $(\boldsymbol{\theta}_n)_n$ be any sequence converging to $\boldsymbol{\theta}_0$ in $\mathcal{C}^2(\overline{\omega})$. Without loss of generality, we can assume that
		$$
		\|\boldsymbol{\theta}_n - \boldsymbol{\theta}_0\|_{\mathcal{C}^2(\overline{\omega})} < \delta_0,
		$$
		where $\delta_0$ is the positive number appearing as in Step 1 of the proof of Part (i). Thus, the sequence $(\boldsymbol{u}_{\boldsymbol{\theta}_n})_n$ is bounded in $\boldsymbol{V}(\omega)$.
		
		Consider any weakly convergent subsequence (still denoted by) $(\boldsymbol{u}_{\boldsymbol{\theta}_n})_n$ and let $\boldsymbol{u}^0 \in \boldsymbol{V}(\omega)$ be its weak limit. We will prove that $\boldsymbol{u}^0$ is a solution of problem \hyperref[plate]{$\mathcal{P}_{plate}$}. First, since
		$$
		\boldsymbol{\theta}_n \to \boldsymbol{\theta}_0 \textrm{ in } \mathcal{C}^2(\overline{\omega}) \textrm{ as } n \to \infty,
		$$
		we easily deduce that, for every $\boldsymbol{u} \in \boldsymbol{V}(\omega)$,
		\begin{equation}
			\label{s3e18}
			J(\boldsymbol{u}) = \lim\limits_{n\to\infty} J_{\boldsymbol{\theta}_n} (\boldsymbol{u}).
		\end{equation}
		
		Next, for every immersion $\boldsymbol{\theta} \in \mathcal{C}^2(\overline{\omega})$ and for every symmetric matrix $(t_{\alpha\beta})$, we have
		\begin{equation}
			\label{s3e19}
			a_{\boldsymbol{\theta}}^{\alpha\beta\sigma\tau} t_{\sigma\tau}t_{\alpha\beta} = \dfrac{4\lambda\mu}{\lambda+2\mu} \big(\sum\limits_{\alpha,\beta} a_{\boldsymbol{\theta}}^{\alpha\beta}t_{\alpha\beta}\big)^2 + 4\mu \textrm{Tr} \big( (a_{\boldsymbol{\theta}}^{\sigma\alpha}t_{\alpha\beta})^T(a_{\boldsymbol{\theta}}^{\sigma\alpha}t_{\alpha\beta})\big),
		\end{equation}
		where $\textrm{Tr}$ denotes the trace operator of square matrices, and $(a_{\boldsymbol{\theta}}^{\sigma\alpha}t_{\alpha\beta})$ is a $2 \times 2$ matrix whose component at its $\sigma$-row and $\beta$-column is $a_{\boldsymbol{\theta}}^{\sigma\alpha}t_{\alpha\beta}$.  
		
		Thanks to the compactness of the embeddings $H^1_0(\omega) \hookrightarrow L^2(\omega)$ and $H^2_0(\omega) \hookrightarrow W^{1,4}(\omega)$, we have
		\begin{equation}
			\label{s3ex1}
			u_{\alpha,\boldsymbol{\theta}_n} \to u^0_{\alpha} \textrm{ in } L^2(\omega) \textrm{ as } n \to \infty, 
		\end{equation}
		and 
		\begin{equation}
			\label{s3ex2}
			u_{3,\boldsymbol{\theta}_n} \to u^0_3 \textrm{ in } W^{1,4}(\omega) \textrm{ as } n \to \infty.
		\end{equation}
		This and the fact that $\boldsymbol{\theta}_n \to \boldsymbol{\theta}_0$ in $\mathcal{C}^2(\overline{\omega})$ imply
		\begin{equation}
			\label{s3e20}                   a_{\boldsymbol{\theta}_n}^{\alpha\beta}E^{\boldsymbol{\theta}_n}_{\alpha\beta}(\boldsymbol{u}_{\boldsymbol{\theta}_n}) \rightharpoonup \delta^{\alpha\beta}E^0_{\alpha\beta}(\boldsymbol{u}^0) \textrm{ in } L^2(\omega),
		\end{equation}
		\begin{equation}
			\label{s3e21}                       a_{\boldsymbol{\theta}_n}^{\alpha\beta}F^{\boldsymbol{\theta}_n}_{\alpha\beta}(\boldsymbol{u}_{\boldsymbol{\theta}_n}) \rightharpoonup \delta^{\alpha\beta}F^0_{\alpha\beta}(\boldsymbol{u}^0) \textrm{ in } L^2(\omega),
		\end{equation}
		\begin{equation}
			\label{s3e22}    a_{\boldsymbol{\theta}_n}^{\sigma\alpha}E^{\boldsymbol{\theta}_n}_{\alpha\beta}(\boldsymbol{u}_{\boldsymbol{\theta}_n}) \rightharpoonup \delta^{\sigma\alpha} E^0_{\alpha\beta}(\boldsymbol{u}^0) \textrm{ in } L^2(\omega),
		\end{equation}
		and
		\begin{equation}
			\label{s3e23}    a_{\boldsymbol{\theta}_n}^{\sigma\alpha}F^{\boldsymbol{\theta}_n}_{\alpha\beta}(\boldsymbol{u}_{\boldsymbol{\theta}_n}) \rightharpoonup \delta^{\sigma\alpha} F^0_{\alpha\beta}(\boldsymbol{u}^0) \textrm{ in } L^2(\omega).
		\end{equation}
		Then, it follows from \eqref{s3e19}-\eqref{s3e23} that
		\begin{equation}
			\label{s3e24}
			\liminf\limits_{n\to\infty} J_{\boldsymbol{\theta}_n}(\boldsymbol{u}_{\boldsymbol{\theta}_n}) \geq J(\boldsymbol{u}^0).
		\end{equation}
		Therefore, by combining \eqref{s3e18} and \eqref{s3e24}, we deduce that
		\begin{align*}
			J(\boldsymbol{u}) &= \lim\limits_{n\to\infty} J_{\boldsymbol{\theta}_n} (\boldsymbol{u}) = \liminf\limits_{n\to\infty} J_{\boldsymbol{\theta}_n} (\boldsymbol{u}) \\
			& \geq \liminf\limits_{n\to\infty} J_{\boldsymbol{\theta}_n} (\boldsymbol{u}_{\boldsymbol{\theta}_n}) \geq J(\boldsymbol{u}^0),
		\end{align*}
		for all $\boldsymbol{u} \in \boldsymbol{V}(\omega)$. Thus, $\boldsymbol{u}^0$ is a minimizer of the functional $J$ over $\boldsymbol{V}(\omega)$.
		
		\noindent\textbf{Step 2:} We will prove that
		\begin{equation}
			\label{s3e25}
			\boldsymbol{u}_{\boldsymbol{\theta}_n} \to \boldsymbol{u}^0 \textrm{ in } \boldsymbol{V}(\omega) \textrm{ as } n \to \infty.
		\end{equation}
		
		To this end, since $\boldsymbol{u}_{\boldsymbol{\theta}_n}$ is minimizer of $J_{\boldsymbol{\theta}_n}$, it is also a critical point of $J_{\boldsymbol{\theta}_n}$. This implies
		\begin{align}
			\label{s3e26}
			\nonumber
			&\int\limits_{\omega} \Big\{  \frac{\varepsilon^3}{3} a_{\boldsymbol{\theta}_n}^{\alpha \beta \sigma \tau} F_{\sigma\tau}^{\boldsymbol{\theta}_n}(\boldsymbol{u}_{\boldsymbol{\theta}_n})F_{\alpha\beta}^{\boldsymbol{\theta}_n}(\boldsymbol{v}) + \varepsilon a_{\boldsymbol{\theta}_n}^{\alpha \beta \sigma \tau} E^{\boldsymbol{\theta}_n}_{\sigma\tau}(\boldsymbol{u}_{\boldsymbol{\theta}_n})E'^{\boldsymbol{\theta}_n}_{\alpha\beta}(\boldsymbol{u}_{\boldsymbol{\theta}_n})(\boldsymbol{v}) \Big\} \sqrt{a_{\boldsymbol{\theta}_n}} \, dy \\
			& = \int\limits_{\omega} p^i v_i \sqrt{a_{\boldsymbol{\theta}_n}} \, dy, \textrm{ for every } \boldsymbol{v} \in \boldsymbol{V}(\omega),
		\end{align}
		where
		\begin{align*}
			E'^{\boldsymbol{\theta}_n}_{\alpha\beta}(\boldsymbol{u}_{\boldsymbol{\theta}_n})(\boldsymbol{v}) := & \dfrac{1}{2} \big(\partial_\beta v_{\alpha} + \partial_\alpha v_{\beta} \big) - \Gamma^\sigma_{\alpha\beta,\boldsymbol{\theta}_n}v_{\sigma} - b_{\alpha\beta,\boldsymbol{\theta}_n} v_{3} \\
			& + \dfrac{1}{2} \partial_\alpha u_{3,\boldsymbol{\theta}_n}\partial_\beta v_{3} + \dfrac{1}{2} \partial_\alpha v_{3}\partial_\beta u_{\boldsymbol{3,\theta}_n},
		\end{align*}
		for every $\boldsymbol{v} \in \boldsymbol{V}(\omega)$.
		
		Next, thanks to \eqref{s3ex1}, \eqref{s3ex2}, and the fact that $\boldsymbol{\theta}_n \to \boldsymbol{\theta}_0 \textrm{ in } \mathcal{C}^2(\overline{\omega})$, we have
		\begin{equation}
			\label{s3e27}
			\int\limits_{\omega} p^i (u_{i,{\boldsymbol{\theta}_n}} - u^0_i) \sqrt{a_{\boldsymbol{\theta}_n}} \, dy \to 0 \textrm{ as } n \to \infty.
		\end{equation}
		Also, recall that
		$$
		\boldsymbol{u}_{\boldsymbol{\theta}_n} \rightharpoonup \boldsymbol{u}^0 \textrm{ in } \boldsymbol{V}(\omega)
		$$
		as $n\to\infty$, we then easily obtain
		\begin{align}
			\label{s3e28}
			\int\limits_{\omega}   \frac{\varepsilon^3}{3} a_{\boldsymbol{\theta}_n}^{\alpha \beta \sigma \tau} F_{\sigma\tau}^{\boldsymbol{\theta}_n}(\boldsymbol{u}^0)F_{\alpha\beta}^{\boldsymbol{\theta}_n}(\boldsymbol{u}_{\boldsymbol{\theta}_n}-\boldsymbol{u}^0)  \sqrt{a_{\boldsymbol{\theta}_n}} \, dy \to 0,    
		\end{align}
		and
		\begin{align}
			\label{s3e29}
			\int\limits_{\omega}   \varepsilon a_{\boldsymbol{\theta}_n}^{\alpha \beta \sigma \tau} E_{\sigma\tau}^{\boldsymbol{\theta}_n}(\boldsymbol{u}^0)E'^{\boldsymbol{\theta}_n}_{\alpha\beta}(\boldsymbol{u}_{\boldsymbol{\theta}_n})(\boldsymbol{u}_{\boldsymbol{\theta}_n}-\boldsymbol{u}^0)  \sqrt{a_{\boldsymbol{\theta}_n}} \, dy \to 0,    
		\end{align}
		as $n\to\infty$. Now, let
		$$
		\boldsymbol{v} = \boldsymbol{u}_{\boldsymbol{\theta}_n} - \boldsymbol{u}^0,
		$$
		in \eqref{s3e26}, it follows from \eqref{s3ex1}, \eqref{s3ex2}, and \eqref{s3e27}-\eqref{s3e29} that
		\begin{align*}
			&\int\limits_{\omega}   \frac{\varepsilon^3}{3} a_{\boldsymbol{\theta}_n}^{\alpha \beta \sigma \tau} \partial_{\sigma\tau}({u}_{3,\boldsymbol{\theta}_n}-{u}^0_3)  \partial_{\alpha\beta}({u}_{3,\boldsymbol{\theta}_n}-{u}^0_3)\sqrt{a_{\boldsymbol{\theta}_n}} \, dy \\
			&+ \int\limits_{\omega}  \varepsilon a_{\boldsymbol{\theta}_n}^{\alpha \beta \sigma \tau} e_{\sigma\tau}(\boldsymbol{u}_{\boldsymbol{\theta}_n}-\boldsymbol{u}^0)  e_{\alpha\beta}(\boldsymbol{u}_{\boldsymbol{\theta}_n}-\boldsymbol{u}^0)\sqrt{a_{\boldsymbol{\theta}_n}} \, dy \\
			&\to 0, \textrm{ as } n \to \infty,
		\end{align*}
		where 
		$$
		e_{\alpha\beta}(\boldsymbol{v}) := \dfrac{1}{2} \big(\partial_\alpha v_\beta + \partial_\beta v_\alpha \big), \textrm{ for all } \boldsymbol{v} \in \boldsymbol{V}(\omega).
		$$
		Therefore, as a consequence of the positive-definiteness of the elasticity tensor $(a_{\boldsymbol{\theta}}^{\alpha\beta\sigma\tau})$ given in \eqref{s2e2} and Korn's inequality, we deduce that
		$$
		\sum\limits_{\alpha,\beta} \|\partial_{\alpha\beta}({u}_{3,\boldsymbol{\theta}_n}-{u}^0_3)\|_{L^2(\omega)} \to 0,
		$$
		and 
		$$
		\sum\limits_{\alpha} \big(\| \partial_{\alpha} (u_{1,\boldsymbol{\theta}_n} - u^0_1) \|_{L^2(\omega)} + \| \partial_{\alpha} (u_{2,\boldsymbol{\theta}_n}-u^0_2) \|_{L^2(\omega)}\big) \to 0,
		$$
		as $n \to \infty$. Hence, we obtain \eqref{s3e25}. Our proof is complete.
	\end{proof}
	
	\begin{rems}
		\label{s3r2}
		\begin{itemize}
			\item[(i)] In Theorem \ref{s3t1}, we have shown that, for every arbitrarily given density $\boldsymbol{p}$, the problem \hyperref[shallow]{$\mathcal{P}_{shallow, \boldsymbol{\theta}}$} admits a solution provided $\boldsymbol{\theta}$ is sufficiently close to $\boldsymbol{\theta}_0$. Notice that our analysis does not involve any rigidity property of shells, only the rigidity of the plates. Nevertheless, it should be emphasized that the question of given a fixed immersion $\boldsymbol{\theta}$, whether \hyperref[shallow]{$\mathcal{P}_{shallow, \boldsymbol{\theta}}$} admits a solution under general applied forces is still unsolved (see also Remark \ref{s3r1}). 
			\item[(ii)] Our result applies to more general boundary conditions, as the rigidity property of plates is established for a broader class of such conditions (see \cite{giamar}). 
			\item[(iii)] The uniqueness of minimizing solutions to the problem \hyperref[plate]{$\mathcal{P}_{plate}$} follows from results in \cite{bern} and \cite{cd}, provided the applied forces are sufficiently small. In subsequent work, we will investigate the uniqueness of solutions for the nonlinear shallow shell model (and, consequently, for the nonlinearly elastic plate model) for a broader class of applied forces, including those of arbitrarily large magnitude.
		\end{itemize}
	\end{rems}

	\section*{Declarations}
	The authors have no conflicts of interest to declare that are relevant to the content of this article.
	
	\section*{Data availability}
	Data sharing is not applicable to this article as no datasets were generated or analyzed during the current study.

	\section*{Acknowledgements} 
	The first author is supported by the grant PRIMUS/24/SCI/020 of Charles University, and within the frame of the project Ferroic Multifunctionalities (FerrMion) [project No. CZ.02.01.01/00/22\_008/0004591], within the Operational Programme Johannes Amos Comenius co-funded by the European Union (JS).

	\newcommand \auth{} 
	\newcommand \jour {} 
	\newcommand \book {}

\end{document}